\documentclass[
 aip,
 amsmath,amssymb,
reprint,%
]{revtex4-1}

\usepackage{graphicx}
\usepackage{dcolumn}
\usepackage{bm}

\usepackage[utf8]{inputenc}
\usepackage[T1]{fontenc}
\usepackage{etoolbox}

\usepackage{subcaption}
\usepackage{mathtools}
\usepackage{amsmath,amssymb}
\usepackage[dvipsnames]{xcolor} 

\newcommand\HL[1]{{\color{black}#1}}

\makeatletter
\def\@email#1#2{%
 \endgroup
 \patchcmd{\titleblock@produce}
  {\frontmatter@RRAPformat}
  {\frontmatter@RRAPformat{\produce@RRAP{*#1\href{mailto:#2}{#2}}}\frontmatter@RRAPformat}
  {}{}
}%
\makeatother

\usepackage[dvipsnames]{xcolor} 

\newcounter{secstmt}[section]

\newif\ifinappendix
\inappendixfalse

\renewcommand{\thesecstmt}{%
  \ifinappendix
    \thesection.\arabic{secstmt}
  \else
    \arabic{section}.\arabic{secstmt}
  \fi}

\makeatletter
\let\origappendix\appendix
\renewcommand{\appendix}{%
  \inappendixtrue
  \origappendix}
\makeatother

\newenvironment{theorem}[1][]%
  {\par\smallskip
   \refstepcounter{secstmt}%
   \noindent\textcolor{MidnightBlue}{{\textsf{Theorem} \thesecstmt%
     \if\relax\detokenize{#1}\relax\else\ \textsf{(#1)}\fi.}}%
   \ \itshape}
  {\par\smallskip\normalcolor\upshape}

\newenvironment{lemma}[1][]%
  {\par\smallskip
   \refstepcounter{secstmt}%
   \noindent\textcolor{MidnightBlue}{{\textsf{Lemma} \thesecstmt%
     \if\relax\detokenize{#1}\relax\else\ \textsf{(#1)}\fi.}}%
   \ \itshape}
  {\par\smallskip\normalcolor\upshape}

\newenvironment{remark}[1][]%
  {\par\smallskip
   \refstepcounter{secstmt}%
   \noindent\textcolor{MidnightBlue}{\textit{Remark} \thesecstmt%
     \if\relax\detokenize{#1}\relax\else\ \textsf{(#1)}\fi.}%
   \ }
  {\par\smallskip\normalcolor\upshape}

\newenvironment{proof}[1][Proof.]%
  {\par\smallskip
   \noindent\textcolor{MidnightBlue}{\textit{#1}}\ }
  {\par\smallskip\normalcolor}
  
\begin{document}


\title[]{Band-Limited Equivalence of Convolution Operators and its Application to Filtered Vorticity Dynamics}
\author{Satori Tsuzuki}
\altaffiliation[]{Advanced Multiscale Fluid Science, Research Center for Advanced Science and Technology, The University of Tokyo, 4-6-1 Komaba, Meguro-ku, Tokyo 153-8904, Japan}
\homepage{https://researchmap.jp/Satori_Tsuzuki?lang=en}
\affiliation{%
The University of Tokyo, 4-6-1, Komaba, Meguro-ku, Tokyo 153-8904, Japan
}%

\date{\today}

\begin{abstract}
In this study, we established a general theorem regarding the equivalence of convolution operators restricted to a finite spectral band. We demonstrated that two kernels with identical Fourier transforms over the resolved band act identically on all band-limited functions, even if their kernels differ outside the band. This property is significant in applied mathematics and computational physics, particularly in scenarios where measurements or simulations are spectrally truncated. As an application, we examine the proportionality relation $S(\bm r) \approx \zeta\,\omega(\bm r)$ in filtered vorticity dynamics and clarify why real-space diagnostics can underestimate the spectral proportionality due to unobservable degrees of freedom. Our theoretical findings were supported by numerical illustrations using synthetic data. 
\end{abstract}

\maketitle

\section{Introduction}
In this study, we address the following question: \emph{When two convolution operators agree on all band-limited functions, must their kernels agree everywhere?} We show that in a band-limited setting, this is not necessarily the case and provide a precise characterization.

In many areas of applied mathematics, ranging from signal processing to computational physics, two convolution operators acting on a spectrally truncated function space are often compared. A recurring question is, if two such operators produce identical outputs for all band-limited inputs, must their kernels be equal everywhere? A common but incorrect assumption is that the equality of integrals for all test functions enforces pointwise equality of the kernels. This study disproves this assumption in a band-limited setting and provides a rigorous characterization of when two convolution operators are equivalent over a restricted spectral band.

Beyond its mathematical significance, this result has practical implications in fields where observations or simulations are spectrally truncated. Examples include imaging systems, seismic inversion, numerical simulations with finite resolution, and large eddy simulation (LES)~\cite{sagaut2006} in fluid dynamics. In the case of LES, unresolved scales are effectively projected out using a numerical filter, so any operator relations are meaningful only on the resolved spectral band.
As a concrete application, we consider filtered vorticity dynamics, which provides a representative setting where band-limited equivalence naturally arises. This choice is motivated by the fact that LES inherently relies on spectrally truncated descriptions~\cite{sagaut2006, smagorinsky1963}, and multi-filter relations~\cite{germano1991} illustrate precisely the type of band-limited convolution equivalence developed in this work. 
In particular, closure assumptions proportional to vorticity have been motivated by linear-response theory~\cite{bolle2024}, statistical analyses of vorticity distributions~\cite{wilczek2009}, and dynamical studies of vortex stretching and enstrophy production~\cite{buaria2020}, further reinforcing the relevance of our framework.

\section{Mathematical Framework}
We work in $\mathbb{R}^d$ with $L^2$ functions, and define the Fourier transform as
\begin{equation}
  \hat{f}(\bm k) = \int_{\mathbb{R}^d} f(\bm r)\, e^{-i \bm k\cdot\bm r}\, \mathrm{d}\bm r .
\end{equation}
Let $\mathcal{B}_\Delta \subset \mathbb{R}^d$ denote the \emph{resolved spectral band} and
$\mathcal{V}_\Delta=\{\omega\in L^2:\operatorname{supp}\hat\omega\subset\mathcal{B}_\Delta\}$.
For kernel $K \in L^1(\mathbb{R}^d)$, we define the convolution operator as
\begin{equation}
  (T_K \omega)(\bm r) = \int_{\mathbb{R}^d} K(\bm r - \bm r')\, \omega(\bm r')\, \mathrm{d}\bm r' .
\end{equation}
Its Fourier symbol is $\hat{K}(\bm k)$ such that $\widehat{T_K \omega}(\bm k) = \hat{K}(\bm k)\,\hat{\omega}(\bm k)$. We used a standard Fourier analysis on $L^1$ and $L^2$ (Plancherel, convolution theorem, and the $L^1\!\to\!L^\infty$ bound to $\hat K$); see~\cite{grafakos2014,katznelson2004,young1912}.

\begin{theorem}[Band-limited equivalence]\label{thm:ble}
Let $K_1,K_2\in L^1(\mathbb{R}^d)$ and $\mathcal{B}_\Delta$ have finite measures.
Then
\begin{align}
  T_{K_1}\omega = T_{K_2}\omega \ \ \forall\,\omega\in\mathcal{V}_\Delta \nonumber \\
  \quad\Longleftrightarrow\quad
  \hat K_1(\bm k)=\hat K_2(\bm k) \ \text{ for a.e.\ }\bm k\in\mathcal{B}_\Delta . \nonumber
\end{align}
\end{theorem}

\begin{proof}
$(\Rightarrow)$ Suppose $T_{K_1}\omega = T_{K_2}\omega$ for all $\omega\in\mathcal{V}_\Delta$.
For $K\in L^1$ and $\omega\in L^2$, Young's inequality
\[
  \|K * \omega\|_{L^2} \ \le\ \|K\|_{L^1} \,\|\omega\|_{L^2}
\]
ensures that $K * \omega \in L^2$. The convolution theorem
\[
  \widehat{K*\omega} = \hat{K}\,\hat{\omega}
\]
holds in the $L^2$-sense (i.e., equality as elements of $L^2(\mathbb{R}^d)$).  
Taking the Fourier transforms and applying this identity, we obtain
\[
  (\hat{K}_1(\bm k)-\hat{K}_2(\bm k))\,\hat{\omega}(\bm k) = 0
  \quad\text{for a.e.\ }\bm k,\ \forall\,\omega\in\mathcal{V}_\Delta .
\]
Let $\phi(\bm k):=(\hat{K}_1(\bm k)-\hat{K}_2(\bm k))\,\mathbf{1}_{\mathcal{B}_\Delta}(\bm k)$.
Because $\mathcal{B}_\Delta$ has a finite measure and $\hat{K}_i\in L^\infty$ for $i=1,2$ (by the $L^1\!\to\!L^\infty$ bound), we have $\phi\in L^2$.
The complex conjugation is denoted by $(\cdot)^{\ast}$.
Based on the unitarity of the Fourier transform on $L^2$, we choose $\hat{\omega} = \phi^{\ast}$, which satisfies $\operatorname{supp}\hat{\omega} \subset \mathcal{B}_\Delta$; thus, $\omega \in \mathcal{V}_\Delta$.
Then
\[
  0=\|(\hat{K}_1-\hat{K}_2)\hat{\omega}\|_{L^2}^2
   = \int_{\mathbb{R}^d} |\phi(\bm k)|^2\, \mathrm{d}\bm k,
\]
so $\phi=0$ a.e., i.e., $\hat{K}_1(\bm k)=\hat{K}_2(\bm k)$ for a.e.\ $\bm k\in\mathcal{B}_\Delta$.

$(\Leftarrow)$ Conversely, if $\hat{K}_1 = \hat{K}_2$ on $\mathcal{B}_\Delta$ and $\omega\in\mathcal{V}_\Delta$,
then $\hat\omega$ vanishes outside $\mathcal{B}_\Delta$ and
$\hat{K}_1(\bm k)\hat{\omega}(\bm k)=\hat{K}_2(\bm k)\hat{\omega}(\bm k)$ for a.e.\ $\bm k$.
Hence, $\widehat{T_{K_1}\omega-T_{K_2}\omega}=0$ in $L^2$ and by Plancherel
$T_{K_1}\omega=T_{K_2}\omega$ in $L^2$ (in particular, \ a.e.).
\end{proof}

\begin{remark}[Band-limited counterexample and interpretation]
When $d=1$, select $\varphi,\psi\in C_c^\infty(\mathbb{R})$ such that
$\varphi\equiv 1$ on $[-k_c,k_c]$, $\operatorname{supp}\varphi\subset[-2k_c,2k_c]$,
and $\operatorname{supp}\psi\subset(-\infty,-3k_c]\cup[3k_c,\infty)$ with $\psi\not\equiv0$.
We set $\widehat{K}_1=\varphi$ and $\widehat{K}_2=\varphi+\psi$.
Then $K_1,K_2=\check{\widehat{K}_1},\check{\widehat{K}_2}\in\mathcal{S}(\mathbb{R})\subset L^1(\mathbb{R})$,
$\widehat{K}_1=\widehat{K}_2$ on $[-k_c,k_c]$ and for every $\omega\in\mathcal{V}_\Delta$,
\[
\widehat{(T_{K_1}-T_{K_2})\omega}=(\widehat{K}_1-\widehat{K}_2)\,\widehat{\omega}=0,
\]
so $T_{K_1}\omega=T_{K_2}\omega$ and $K_1\neq K_2$ in the physical space.
This demonstrates that differences outside the resolved band are unobservable on $\mathcal{V}_\Delta$.
\end{remark}

\section{Application to Filtered Vorticity Dynamics}
In the LES context, unresolved high-wavenumber degrees-of-freedom create distinct real-space kernels that are indistinguishable when restricted to a resolved spectral band (cf.\ Theorem~\ref{thm:ble}).

Large-eddy simulation of incompressible flow advances the filtered vorticity $\overline{\omega}$
related to the unfiltered field as follows:
\begin{equation}
  \overline{\omega}(\bm r)
  = \int_{\mathbb{R}^d} F(\bm r-\bm r')\,\omega(\bm r')\,\mathrm{d}\bm r',
\end{equation}
where $F$ denotes the LES filter kernel. We idealize $F$ as a linear, translation-invariant
spatial filter, i.e., $\overline{\phi} = F * \phi$ with $F \in L^1(\mathbb{R}^d)$ and
$\int_{\mathbb{R}^d}F = 1$ such that $\overline{\partial_j \phi} = \partial_j \overline{\phi}$
(derivatives commute through filtering; see~\cite{sagaut2006}).
This idealization holds exactly in periodic domains or on $\mathbb{R}^d$ and may only be
approximate near physical boundaries.

Starting from the incompressible vorticity transport equation
\[
  \partial_t \omega(\bm r,t)
  = \eta\nabla^2\omega(\bm r,t) - (\bm u\!\cdot\!\nabla)\omega(\bm r,t)
\]
$\bigl[\text{in 3D add the stretching term }-(\bm\omega\!\cdot\!\nabla)\bm u\bigr]$, and applying filter $\overline{(\,\cdot\,)} = F*\!(\,\cdot\,)$ with a constant viscosity, we obtain
\begin{align}
  \partial_t\,\overline{\omega}(\bm r,t)
  &= \eta\nabla^2\overline{\omega}(\bm r,t) + S(\bm r,t), \nonumber \\
  S(\bm r,t) &:= -\,\overline{(\bm u\!\cdot\!\nabla)\,\omega}(\bm r,t)
  \label{eq:filtered-vort-source}
\end{align}
$\bigl[ \text {add} -\,\overline{(\bm\omega\!\cdot\!\nabla)\bm u}\ \text{in 3D} \bigr]$, where $S$ represents the subgrid-scale source term arising from the non-commutativity of
nonlinear advection and filtering. 
In the following, we \HL{consider} the 2D case in which the vortex-stretching term vanishes because
vorticity is orthogonal to the flow plane and its magnitude is preserved in incompressible flows.
Motivated by the scale-similarity and symmetry arguments (see Appendix~\ref{app:closure-derivation}),
a few closure models approximate the subgrid term as
\[
  S(\bm r) \approx \zeta\,\omega(\bm r).
\]
This proportional form has also been supported by linear-response theory of vortex dynamics~\cite{bolle2024},
statistical analyses of vorticity distributions in turbulence~\cite{wilczek2009},
and recent high-Reynolds-number studies of vortex stretching and enstrophy production~\cite{buaria2020}. A complementary demonstration was recently provided by freely decaying two-dimensional turbulence initialized with a concentrated point-vortex distribution, in which the scale-to-scale transfer mediated by $\nabla\times\omega$ was directly observed in numerical simulations~\cite{tsuzuki2025}.

Combining this ansatz with Green's function $G$ of the linear filtered operator yields
\begin{equation}
  \overline{\omega}(\bm r)
  = \int_{\mathbb{R}^d} \zeta\,G(\bm r-\bm r')\,\omega(\bm r')\,\mathrm{d}\bm r'.
  \label{eq:selfconv}
\end{equation}
Here $\omega$ denotes a vorticity field restricted to the resolved band $\mathcal{B}_\Delta$, so that \eqref{eq:selfconv} is consistent with the band-limited framework of Theorem~\ref{thm:ble}. This restriction ensures that any comparison between $F$ and $\zeta G$ is meaningful only on the resolved subspace. Recasting the difference between two convolution integrals as the difference between their kernels requires that the test function $\omega$ span the full function space. In LES, however, while the physical vorticity field prior to filtering contains arbitrarily fine scales, the effective $\omega$ entering relation \eqref{eq:selfconv} is truncated to the after-filter resolved band $\mathcal{B}_\Delta$. Thus, unresolved degrees of freedom inevitably remain outside $\mathcal{B}_\Delta$. This restriction arises naturally in multi-filter formulations~\cite{germano1991}, where comparisons are only meaningful within the overlap of resolved spectra, and it also reflects the finite resolution of numerical simulations, which truncate all fields to a band-limited representation. Consequently, $\omega$ cannot probe degrees of freedom outside $\mathcal{B}_\Delta$, and the difference of kernels cannot be directly inferred. In particular,
\[
   (F-\zeta G)*\omega \equiv 0 \quad \forall \,\omega\in\mathcal{B}_\Delta
   \;\;\not\!\!\implies\;\; F \equiv \zeta G .
\]
This means that the situation should not be misinterpreted as admitting only the trivial solution $\omega\equiv 0$; rather, $F$ and $\zeta G$ are \emph{unidentifiable} outside the resolved band, precisely because unresolved components of the physical field are invisible to the test space. Theorem~\ref{thm:ble} shows that only the Fourier transforms need coincide on $\mathcal{B}_\Delta$; differences in the kernels outside this band are unobservable in LES. As later illustrated by numerical tests, this also helps explain why real-space correlations between $S$ and $\omega$ can appear weak even when strong proportionality exists within the resolved band, since unresolved components contaminate pointwise comparisons in the physical space.

In filtered dynamics, the closure relation \eqref{eq:selfconv} can be recast in spectral space
into the recursive form
\begin{equation}
    i\bm{k} \times \widehat{\overline{\omega}}(\bm{k})
    = \widehat{\mathcal{H}}(\bm{k})\,\hat{\omega}(\bm{k})
      + \widehat{\mathcal{F}}(\bm{k})\,
        \big[\,i\bm{k}\times\hat{\omega}(\bm{k})\,\big], 
    \label{eq:reccursiverelation}
\end{equation}
where $\widehat{\mathcal{H}}$ represents the direct filtered contribution and
$\widehat{\mathcal{F}}$ encodes subgrid feedback onto the curl of the unfiltered vorticity.
The second term feeds back a transformed version of $i\bm{k}\times\hat{\omega}$ into itself,
creating a \emph{recursive} dependence of the resolved-scale dynamics on itself through
unresolved scales.
This feedback mechanism exists even for linear operators and produces an intrinsic
small-to-large scale coupling.
In real space, such a mechanism can be obscured by the convolution form of filtering. 
However, in spectral space, the recurrence relation \eqref{eq:reccursiverelation} makes the coupling
explicit and analytically transparent
(see Appendix~\ref{appendix:derivrecurcur} for the derivation).

\begin{remark}[On $\nabla\times\omega$ and multi-scale structure]
Let $\omega(\bm r)$ be the vorticity field and $\hat{\omega}(\bm k)$ its Fourier transform.
The curl operator $\nabla\times$ acts in the spectral space as
\[
    \widehat{\nabla\times\omega}(\bm k) = i\,\bm k \times \hat{\omega}(\bm k),
\]
whose magnitude scales to $|\bm k|\,\|\hat{\omega}(\bm k)\|$.
Even if a proportionality $\hat{S}(\bm k) \approx \zeta\,\hat{\omega}(\bm k)$ holds on the resolved band $\mathcal{B}_\Delta$,
the corresponding curl relation is
\[
    \widehat{\nabla\times S}(\bm k) \approx \zeta\, i\,\bm k \times \hat{\omega}(\bm k),
\]
which introduces an explicit $|\bm k|$--dependent gain.

Consequently, the transfer function measured on $\nabla\times(\cdot)$ remains flat in magnitude only up to the deterministic $|\bm k|$ factor.
Our spectral diagnostics explicitly account for this scaling, ensuring that proportionality can still be robustly identified using curl-based measures.
\end{remark}

\begin{remark}[Extension to 3D and anisotropic filters]
While our explicit vorticity-based application has been presented in a 2D setting
for clarity, the band-limited equivalence theorem itself is dimension independent
and applies equally to convolution kernels $K:\mathbb{R}^3\to\mathbb{R}$ and to
vector-valued operators. In three dimensions, the filtered vorticity equation
contains the additional vortex-stretching term, but once the combined
contribution of all unresolved interactions is modeled as a translation-invariant
linear operator $T$ acting on the resolved vorticity, its action on a
band-limited field is still a Fourier multiplier:
\[
\widehat{T\boldsymbol{\omega}}(\boldsymbol{k})
   = \widehat{H}(\boldsymbol{k})\,\widehat{\boldsymbol{\omega}}(\boldsymbol{k}),
   \qquad \boldsymbol{k}\in \mathcal{B}_\Delta.
\]
Thus, on the resolved band $\mathcal{B}_\Delta$, the closure has exactly the structure
treated in Appendix~\ref{app:closure-derivation}: if $\widehat{H}(\boldsymbol{k})$ is approximately
isotropic on $\mathcal{B}_\Delta$ and varies slowly over the energetic range
$|\boldsymbol{k}|\in[k_1,k_2]\subset(0,k_c]$, then
\[
\widehat{T\boldsymbol{\omega}}(\boldsymbol{k})
   \approx \zeta\,\widehat{\boldsymbol{\omega}}(\boldsymbol{k}),
   \qquad \boldsymbol{k}\in \mathcal{B}_\Delta,
\]
which yields the band-wise proportionality
$T\boldsymbol{\omega}\approx \zeta\,\boldsymbol{\omega}$ for the resolved
vorticity. This conclusion is insensitive to whether the underlying 3D dynamics
includes vortex stretching, since all unresolved contributions are absorbed into
the effective multiplier $\widehat{H}(\boldsymbol{k})$.

Many practical LES implementations also employ anisotropic grid filters. As long
as these filters remain (discrete) convolutions on uniform grids, they admit an
anisotropic Fourier-multiplier representation on rectangular spectral supports
\cite{lund1997,meneveau2000}. The band-limited equivalence theorem then applies
to each scalar entry of $\widehat{H}(\boldsymbol{k})$ on the resolved set
$\mathcal{B}_\Delta$.

Modern subgrid-scale closures, including dynamic models of the
Germano--Piomelli--Moin--Cabot type~\cite{germano1991}, make explicit use of
relations involving both the grid and test filters. These identities are strictly
meaningful only on the intersection of the corresponding resolved spectra. Our
framework provides a rigorous basis for interpreting such relations as band-wise
equivalences: different kernels may disagree substantially outside the common
band while remaining indistinguishable on the doubly filtered subspace. This
perspective further supports the use of spectral gain and coherence as diagnostics
for validating 3D closures in the presence of vortex stretching and anisotropy.
\end{remark}

\section{Numerical Illustrations}
\subsection{Band-limited counterexample}\label{subsec:counterexample}
We illustrate this theorem by using a synthetic example of a periodic square
$\Omega=[0,2\pi]^2$ discretized by an $n\times n$ Fourier collocation grid ($n=256$ unless
stated).  Let $k_c>0$ and denote the resolved band by $\mathcal{B}_\Delta=\{k:\,|k|\le k_c\}$.
A real-valued, band-limited field is generated by prescribing random phases on the band and
zero elsewhere,
\[
  \hat\omega(k)=\exp\{i\phi(k)\}\,\mathbf{1}_{\{|k|\le k_c\}},\qquad \phi(k)\sim\mathrm{Unif}(0,2\pi),
\]
followed by an inverse FFT to obtain $\omega$ in physical space.

We compared the following two convolution symbols:
\[
  \hat F(k)=\mathbf{1}_{\{|k|\le k_c\}},\qquad
  \widehat{\zeta G}(k)=\hat F(k)+Q(k),
\]
where $Q$ is supported on the annulus $2k_c<|k|<3k_c$. Since $\hat\omega$ vanishes
precisely where $Q\neq 0$, both operators act identically on $\omega$:
$(\zeta G*\omega)=(F*\omega)$ on the grid up to roundoff, in agreement with the
band-limited equivalence theorem. Numerically, we observe machine-precision agreement
in both $L^2$ and $L^\infty$ norms (double precision); see Fig.~\ref{fig:bandlimited-counterexample}c.

For visualization, panels~\ref{fig:bandlimited-counterexample}a--b exhibit
$\hat F$ and $\widehat{\zeta G}$ on the $(k_x,k_y)$-plane (only the half-plane $k_y\ge 0$
is shown; the other half follows Hermitian symmetry
$\hat\omega(-k_x,-k_y)=\overline{\hat\omega(k_x,k_y)}$ for a real-valued $\omega$).
The dashed circle marks the band boundary $|k|=k_c$ defining $\mathcal{B}_\Delta$,
whereas the dotted circles at $|k|=2k_c,\,3k_c$ delimit the annulus, indicating the support of $Q$.
Panel~(c) shows a log-scale histogram of the pointwise difference
$\lvert(\zeta G*\omega)-(F*\omega)\rvert$, which clusters at machine precision.

\begin{figure*}[t]
  \centering
  \begin{subfigure}[b]{0.33\textwidth}
    \centering
    \includegraphics[width=\linewidth]{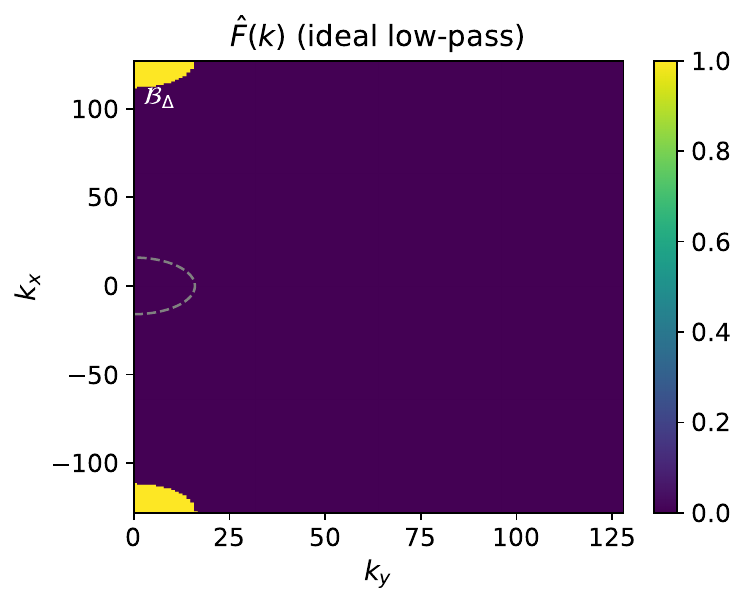}
    \subcaption{}\label{fig:bl-a}
  \end{subfigure}\hfill
  \begin{subfigure}[b]{0.33\textwidth}
    \centering
    \includegraphics[width=\linewidth]{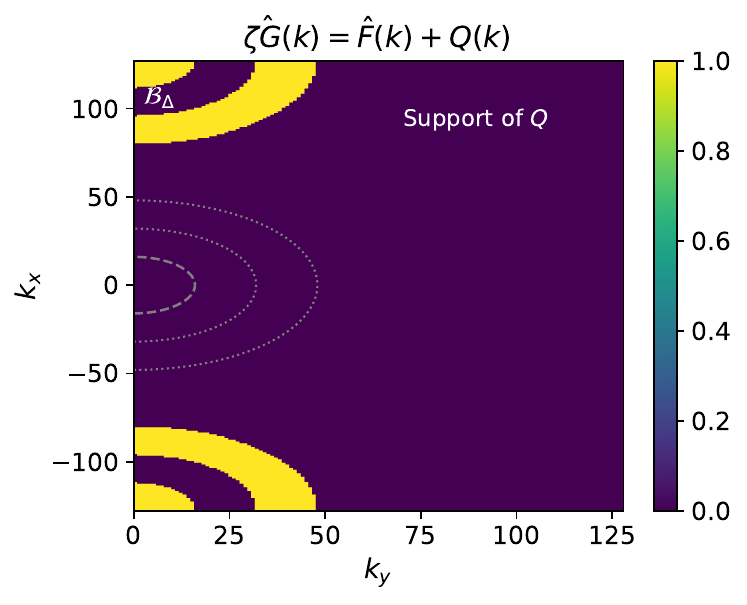}
    \subcaption{}\label{fig:bl-b}
  \end{subfigure}\hfill
  \begin{subfigure}[b]{0.33\textwidth}
    \centering
    \includegraphics[width=\linewidth]{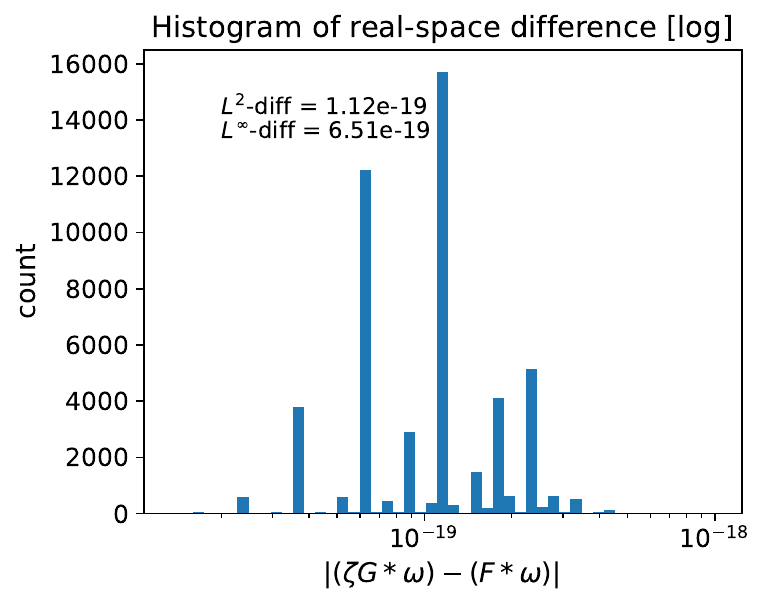}
    \subcaption{}\label{fig:bl-c}
  \end{subfigure}
  \caption{Band-limited counterexample. (a) Ideal low-pass transfer $\hat{F}(k)$.
  (b) Modified transfer $\widehat{\zeta G}(k)=\hat{F}(k)+Q(k)$ that differs from $\hat{F}$ only outside the resolved band.
  (c) Histogram of the real-space difference $\lvert(\zeta G*\omega)-(F*\omega)\rvert$ (log scale).
  \textbf{Dashed circle}: band boundary $\lvert k\rvert = k_c$ defining $\mathcal{B}_\Delta$.
  \textbf{Dotted circles}: radii $\lvert k\rvert = 2k_c,\,3k_c$, delimiting the annular domain indicating the support of $Q$
  (i.e.\ $2k_c < \lvert k\rvert < 3k_c$).
  Only the half-plane $k_y \ge 0$ is shown; the other half follows from the Hermitian symmetry
  $\hat{\omega}(-k_x,-k_y)=\overline{\hat{\omega}(k_x,k_y)}$ for real-valued $\omega$.
  For the band-limited $\omega$, the two outputs coincide up to machine precision (see norms in panel (c)).}
  \label{fig:bandlimited-counterexample}
\end{figure*}

\subsection{Teacher--student spectral tests}\label{subsec:teacher-student}
We assessed the proportionality ansatz on the synthetic data constructed on a periodic square.
Let $\mathcal{B}_\Delta=\{\bm k:\,|\bm k|\!\in[k_1,k_2]\}$ be the resolved band. We generate a
real-valued band-limited field $\omega$ by assigning random phases to $\hat\omega(\bm k)$ on
$\mathcal{B}_\Delta$ and zero elsewhere, and define the “student” field $S$ through an assigned
multiplier on the band, $\hat S(\bm k)=H_{\rm true}(\bm k)\,\hat\omega(\bm k)$ with
$H_{\rm true}(\bm k)\approx \zeta$ on $\mathcal{B}_\Delta$. Optionally, we added (i) band-limited
additive complex noise to $\hat S$ or (ii) a nonlocal distortion, making $H_{\rm true}$ $|\bm k|$–dependent.

The diagnostics are computed using the shell averages over the concentric annuli in $|\bm k|$:
\begin{align}
  H(k)\;&=\;\frac{\langle \hat S(\bm k)\,\overline{\hat\omega(\bm k)}\rangle_{\text{shell}}}
                 {\langle \hat\omega(\bm k)\,\overline{\hat\omega(\bm k)}\rangle_{\text{shell}}},\nonumber \\
  \gamma^2(k)\;&=\;\frac{\big|\langle \hat S(\bm k)\,\overline{\hat\omega(\bm k)}\rangle_{\text{shell}}\big|^2}
                       {\langle|\hat S(\bm k)|^2\rangle_{\text{shell}}\,
                        \langle|\hat\omega(\bm k)|^2\rangle_{\text{shell}}}.
\end{align}
For comparison, we also report a zero-intercept least-squares fit in physical space (scatter of $S$ vs.\ $\omega$). Gain and magnitude-squared coherence are computed via standard cross-spectral estimators; see \cite{bendat2010}.

Across the three scenarios (Fig.~\ref{fig:teacher-student-3x3}), we observed:
\emph{Clean}: \; $|H(k)|\approx\zeta$ (flat on $[k_1,k_2]$), $\arg H(k)\approx 0$, $\gamma^2(k)\approx 1$, and
the real-space fit attained $R^2\approx 1$.
\emph{Noise}:\; $|H(k)|$ remained (nearly) unbiased while its variance and $1-\gamma^2(k)$ increase with decreasing SNR,
leading to a mild drop in the real-space $R^2$.
\emph{Nonlocal}: \; $|H(k)|$ developed a clear $k$–dependence and may exhibit phase offsets; despite relatively high $\gamma^2(k)$,
the single-scalar proportionality degrades in real space, and $R^2$ decreases.
These results confirm that spectral diagnostics (gain and coherence) can detect the proportionality of $\mathcal{B}_\Delta$
more robustly than real-space correlations.

\begin{figure*}[t]
  \centering
  \begin{subfigure}[b]{0.32\textwidth}
    \centering
    \includegraphics[width=\linewidth]{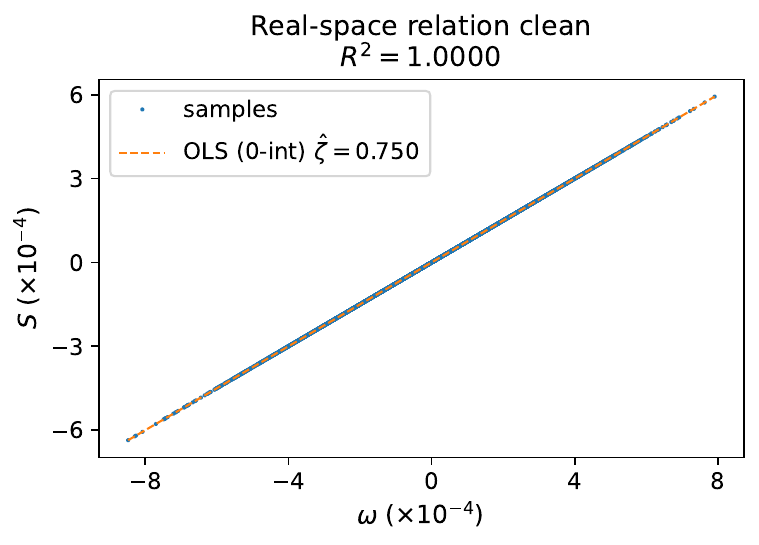}
    \subcaption{Clean}\label{fig:ts-scatter-clean}
  \end{subfigure}\hfill
  \begin{subfigure}[b]{0.32\textwidth}
    \centering
    \includegraphics[width=\linewidth]{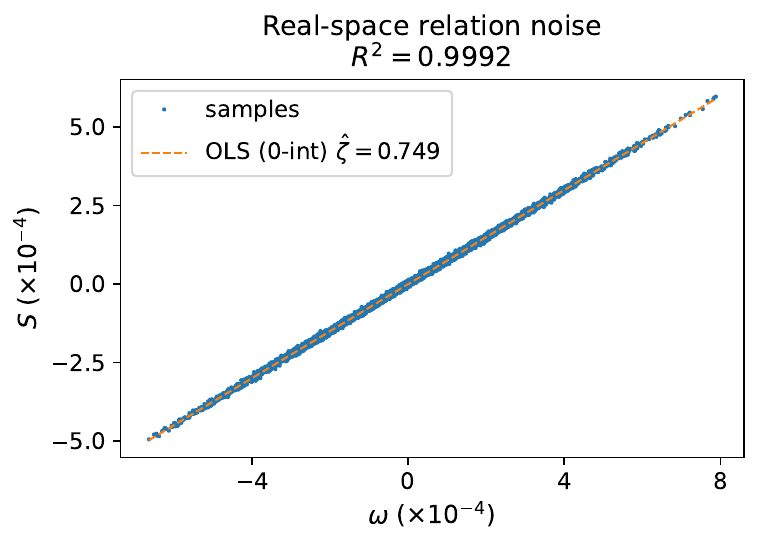}
    \subcaption{Noise}\label{fig:ts-scatter-noise}
  \end{subfigure}\hfill
  \begin{subfigure}[b]{0.32\textwidth}
    \centering
    \includegraphics[width=\linewidth]{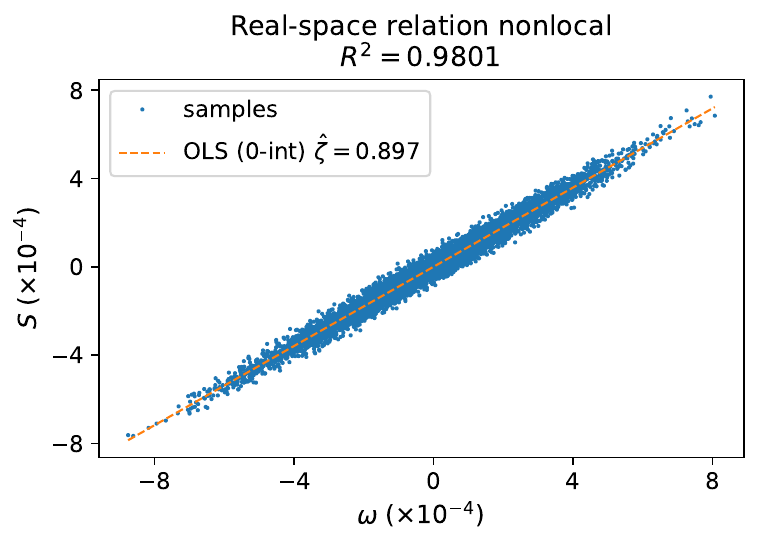}
    \subcaption{Nonlocal}\label{fig:ts-scatter-nonlocal}
  \end{subfigure}

  \vspace{0.6em}

  \begin{subfigure}[b]{0.32\textwidth}
    \centering
    \includegraphics[width=\linewidth]{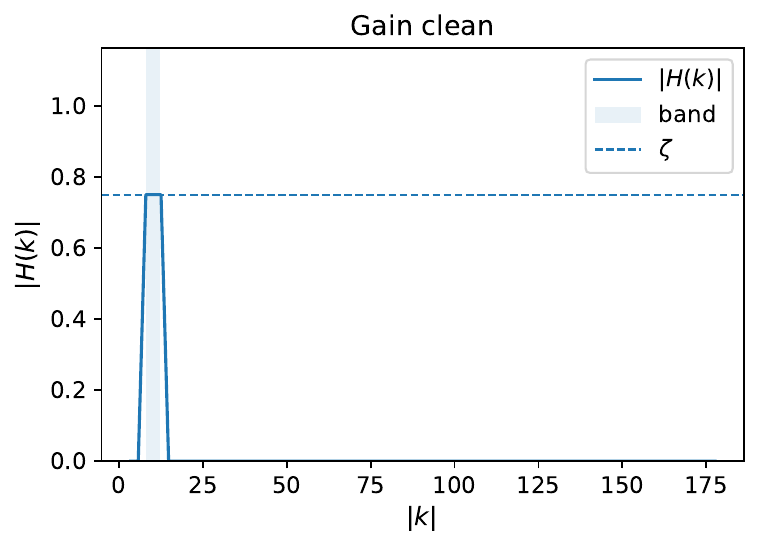}
    \subcaption{$|H(k)|$ (Clean)}\label{fig:ts-gain-clean}
  \end{subfigure}\hfill
  \begin{subfigure}[b]{0.32\textwidth}
    \centering
    \includegraphics[width=\linewidth]{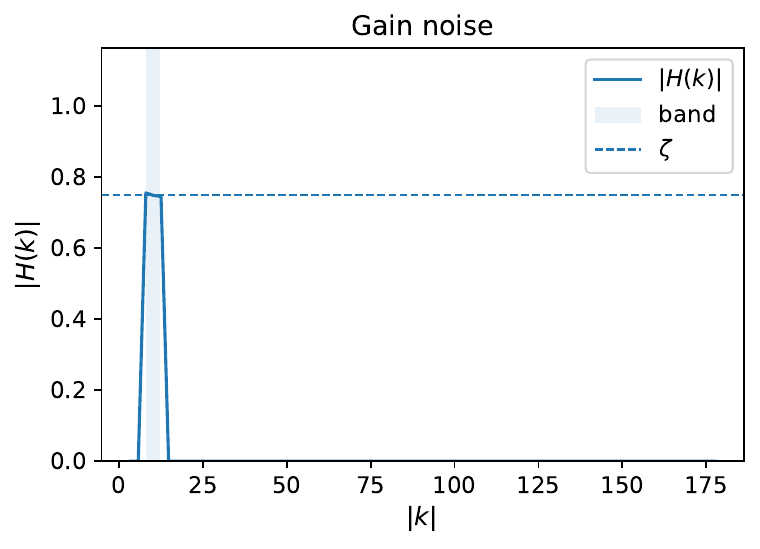}
    \subcaption{$|H(k)|$ (Noise)}\label{fig:ts-gain-noise}
  \end{subfigure}\hfill
  \begin{subfigure}[b]{0.32\textwidth}
    \centering
    \includegraphics[width=\linewidth]{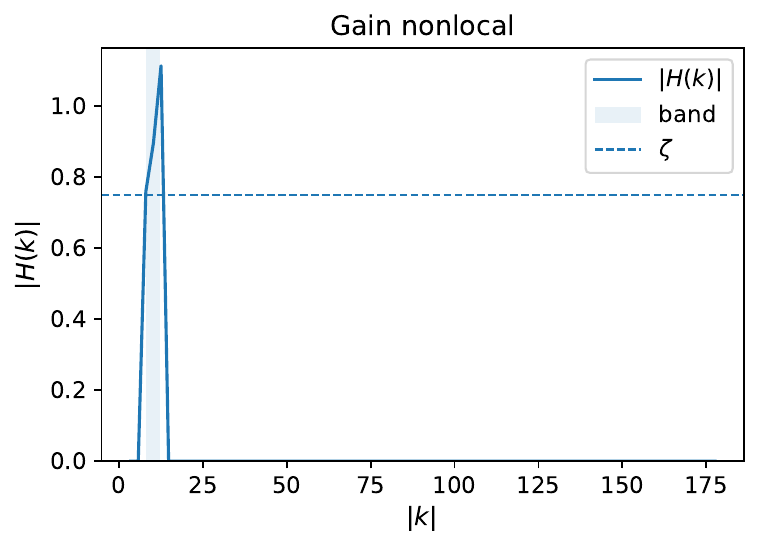}
    \subcaption{$|H(k)|$ (Nonlocal)}\label{fig:ts-gain-nonlocal}
  \end{subfigure}

  \vspace{0.6em}

  \begin{subfigure}[b]{0.32\textwidth}
    \centering
    \includegraphics[width=\linewidth]{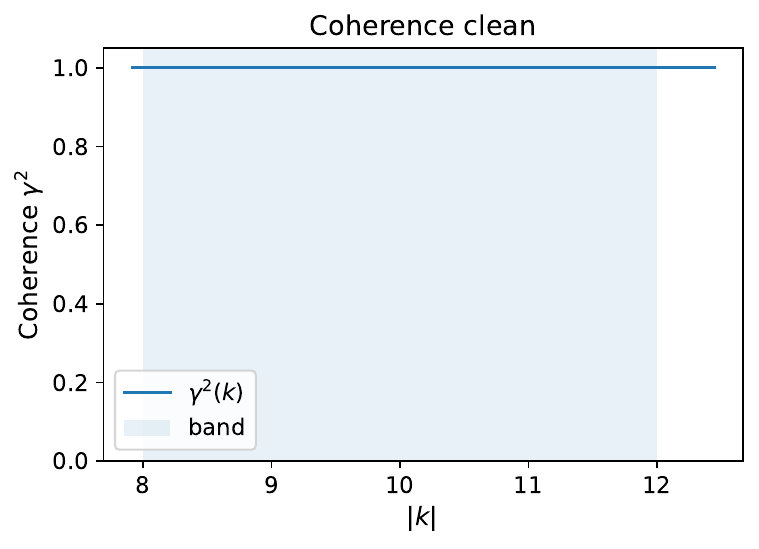}
    \subcaption{$\gamma^2(k)$ (Clean)}\label{fig:ts-coh-clean}
  \end{subfigure}\hfill
  \begin{subfigure}[b]{0.32\textwidth}
    \centering
    \includegraphics[width=\linewidth]{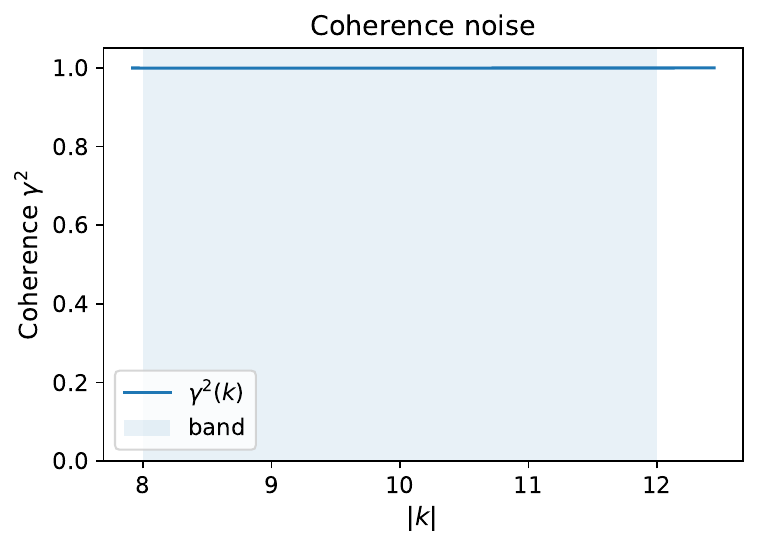}
    \subcaption{$\gamma^2(k)$ (Noise)}\label{fig:ts-coh-noise}
  \end{subfigure}\hfill
  \begin{subfigure}[b]{0.32\textwidth}
    \centering
    \includegraphics[width=\linewidth]{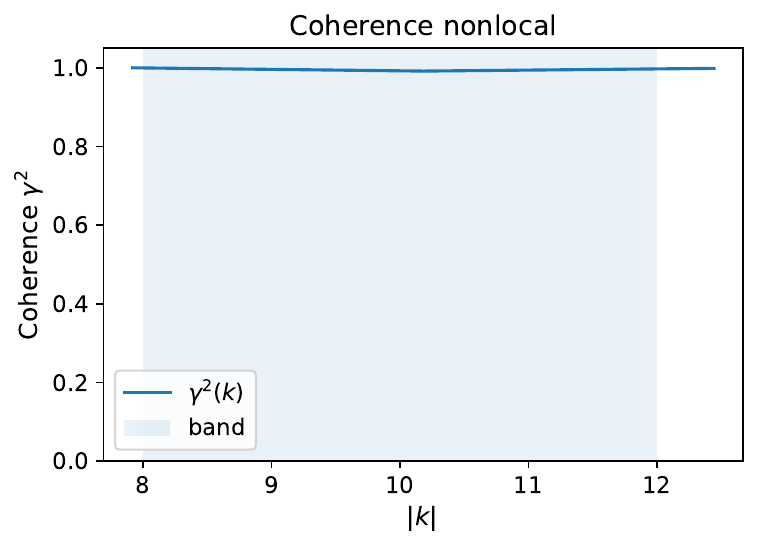}
    \subcaption{$\gamma^2(k)$ (Nonlocal)}\label{fig:ts-coh-nonlocal}
  \end{subfigure}

  \caption{Teacher--student tests on a periodic square ($n=256$, $L=2\pi$) with a band
  $k\in[k_1,k_2]=[8,12]$ and target gain $\zeta=0.75$.
  \textbf{Top row}: real-space scatter of $S$ vs.\ $\omega$ with zero-intercept OLS fit (slope $\hat\zeta$ and $R^2$ shown; axes displayed in scaled units as indicated).
  \textbf{Middle row}: shell-averaged gain magnitude $|H(k)|$; the dashed horizontal line marks $\zeta$ and the band $[k_1,k_2]$ is shaded.
  \textbf{Bottom row}: magnitude-squared coherence $\gamma^2(k)$.
  Three scenarios are shown: \emph{Clean} (ideal proportionality), \emph{Noise}
  (additive band-limited complex white noise with standard deviation $0.02$ on $S$),
  and \emph{Nonlocal} ($k$-dependent gain with {\tt nonlocal\_gain}$=0.6$).
  Estimates use shell-wise least squares,
  $H(k)=\langle \hat S\,\overline{\hat\omega}\rangle/\langle \hat\omega\,\overline{\hat\omega}\rangle$ and
  $\gamma^2(k)=\frac{|\langle \hat S\,\overline{\hat\omega}\rangle|^2}{\langle|\hat S|^2\rangle\,\langle|\hat\omega|^2\rangle}$.
  For comparability, $|H|$ and $\gamma^2$ share common $y$-axis limits across the row.}
  \label{fig:teacher-student-3x3}
\end{figure*}

\section{Discussion and conclusions}\label{sec:disc-conc}
We established a band-limited equivalence theorem for convolution operators and demonstrate how
it explains a common pitfall in the filtered dynamics: two kernels that differ outside the resolved
band are indistinguishable from any operator acting on the band-limited data. This clarifies why the
proportionality $S(\bm r)\approx\zeta\,\omega(\bm r)$ can be robustly identified using spectral
diagnostics, while appearing to be degraded in real-space correlations.

\paragraph*{Implications.}
In any setting in which data or models are spectrally truncated (inverse problems, imaging,
reduced-order modelling, and LES), the equality tests between operators must be confined to the
resolved band. In particular, for LES, our results justify the validation of closures via spectral gain
and coherence rather than (or in addition to) real-space scatter plots, which explains why
discrepancies in the physical space need not contradict the band-wise proportionality.

\paragraph*{Scope and limitations.}
Our analysis targets translation-invariant operators in periodic domains and focuses on
band-limited settings with isotropic shell averages. Numerical illustrations use synthetic
data with band-pass construction, simple $k$-dependent distortions, and additive noise.
These choices isolate the mechanism but do not exhaust all sources of model--data mismatch
(e.g.\ anisotropy, irregular bands, boundaries, sampling effects, or strongly nonstationary
or inhomogeneous flows~\HL{\cite{mason1992}}). 

\HL{In particular, Theorem~2.1 applies strictly to convolution operators on $\mathbb{R}^d$ or on periodic domains, for which the filter can be represented as a global convolution kernel and hence as a Fourier multiplier. Near solid walls, in complex geometries, or for strongly
grid-dependent numerical filters, the effective LES operator need not be translation
invariant \cite{lund1997} and cannot be reduced to a scalar transfer function on the full
domain. In such settings, band-limited equivalence in the sense of Theorem~2.1 is not
expected to hold globally, and our results should be interpreted as describing the idealized
homogeneous or periodic portions of the flow. Nonetheless, away from boundaries or in approximately homogeneous regions, many practical LES filters are either defined as homogeneous convolution kernels or are designed to approximate such convolutions, and the present framework can still serve as a useful conceptual guide for designing and interpreting band-aware diagnostics in those regions.}

\paragraph*{Robustness to approximate band-limitation and noise.}
\HL{The exact statement of Theorem~2.1 presumes perfectly band-limited inputs and exact
coincidence of the Fourier symbols on the resolved band. In practice, numerical filters
possess finite transition regions, FFT grids introduce spectral leakage, and both $S$ and
$\omega$ may be contaminated by noise. In such situations, the difference
$(T_{K_1}-T_{K_2})\omega$ can be bounded in $L^2$ norm by a contribution proportional to the
mismatch of the symbols on the resolved band and by a leakage term proportional to the
energy of $\widehat{\omega}$ outside the band; see Appendix~\ref{app:approx} for a quantitative estimate. This shows that band-wise equivalence remains stable as long as the spectral mismatch
$\|\widehat{K}_1 - \widehat{K}_2\|_{L^\infty(\mathcal{B}_\Delta)}$ is small on $\mathcal{B}_\Delta$ and the
out-of-band content $\|\widehat{\omega}\|_{L^2(\mathcal{B}_\Delta^c)}$ is weak.

Our teacher--student tests in Fig.~\ref{fig:teacher-student-3x3} illustrate this robustness numerically. In the
presence of band-limited additive noise, the estimated gain $H(k)$ remains nearly unbiased
on the band while the magnitude-squared coherence $\gamma^2(k)$ decays as the signal-to-noise
ratio decreases, and real-space scatter plots exhibit a more pronounced degradation. When a
nonlocal, $k$-dependent distortion is introduced, $|H(k)|$ clearly departs from a constant
on the band even though $\gamma^2(k)$ can remain relatively high. These observations are
consistent with the approximate band-limited equivalence and highlight that spectral gain
and coherence provide a more robust assessment of proportionality on $\mathcal{B}_\Delta$ than
real-space correlations in the presence of noise or imperfect filtering.}

\paragraph*{Outlook.}
The framework naturally extends to (i) other closure forms (e.g., eddy viscosity and higher-order
Fourier multipliers on the band), (ii) anisotropic or geometry-induced filters, (iii) finite-sample
effects and uncertainty quantification for $H(k)$ and coherence, and (iv) validation on high-resolution
LES/DNS or experimental datasets. We expect that ``band-aware'' diagnostics will remain decisive even
under moderate noise or nonlocal distortions, as corroborated by our teacher--student tests.

\appendix
\section{Derivation of $S(\bm r)\approx \zeta\,\omega(\bm r)$ under band-limited symmetry assumptions}\label{app:closure-derivation}

We modeled the subgrid contribution $S=T\omega$ as a linear operator $T$ acting on the resolved vorticity.
Assume:
\begin{itemize}
\item[(A1)] (\emph{Band limitation}) $\hat\omega(\bm k)=0$ for $|\bm k|>k_c$; $T$ maps $\mathcal V_\Delta$ to itself.
\item[(A2)] (\emph{Translation invariance}) $T$ is a convolution: $\widehat{T\omega}(\bm k)=\hat H(\bm k)\,\hat\omega(\bm k)$.
\item[(A3)] (\emph{Isotropy on the band}) $\hat H(\bm k)=h(|\bm k|)$ for $|\bm k|\le k_c$.
\item[(A4)] (\emph{Scale separation / narrow band}) On the energetic support $|\bm k|\in[k_1,k_2]\subset(0,k_c]$,
$h(|\bm k|)$ is approximately constant: $h(|\bm k|)=\zeta+o(1)$.
\end{itemize}
Then, for all resolved $\omega$,
\[
\widehat{S}(\bm k)=\hat H(\bm k)\,\hat\omega(\bm k)\approx \zeta\,\hat\omega(\bm k)\quad (|\bm k|\le k_c),
\]
hence $S(\bm r)\approx \zeta\,\omega(\bm r)$ in $\mathcal V_\Delta$.
\medskip

\begin{remark}
(A2) reduces $T$ to a Fourier multiplier; (A3) forces the symbol to depend only on $|\bm k|$; (A4) yields band-wise constancy.
Therefore, the only admissible isotropic, translation-invariant zeroth-order closure on a narrow resolved band is a scalar multiple of the identity.
\end{remark}

\subsection{Heuristic estimate via eddy viscosity on a narrow band}
Starting from the filtered vorticity transport, the nonlinear term produces
$S\approx -\nu_t \nabla^2 \omega$ for several LES closures. For a band-limited $\omega$ with dominant wavenumbers
$|\bm k|\in[k_1,k_2]$, we have in the Fourier space
$\widehat{S}(\bm k)\approx \nu_t |\bm k|^2 \hat\omega(\bm k)$.
If the band is narrow, a representative $k_\star\in[k_1,k_2]$ is used to obtain
\[
S\approx \nu_t k_\star^2\,\omega \quad\Longrightarrow\quad \zeta \approx \nu_t k_\star^2,
\]
with a relative error $O\!\left((k_2-k_1)/k_\star\right)$. This provides a concrete scaling for $\zeta$ when the eddy viscosity picture is appropriate.

\subsection{Probabilistic linear-response estimate ($\zeta \approx 1/\tau$)}
The complementary derivation exhibited a linear response (quasi–Markovian) viewpoint. 
Assuming that the coarse-grained vorticity obeys after filtering and projection onto the resolved subspace,
a Langevin-type closure at leading order,
\begin{equation}
    \partial_t \omega_\Delta(\bm r,t) \;\approx\; -\,\zeta\,\omega_\Delta(\bm r,t) \;+\; \xi(\bm r,t),
    \label{eq:LRP_ODE}
\end{equation}
where $\xi$ collects the fast (unresolved) fluctuations. Under stationarity and weak-coupling assumptions 
(standard in projection-operator/linear-response closures), the proportionality $\hat S(\bm k)\approx \zeta\,\hat\omega(\bm k)$
on the resolved band implies that the decay rate of the (temporal) autocorrelation determines the gain:
\begin{align}
    C_{\omega\omega}(t) \;=\; \langle \omega_\Delta(\cdot,t)\,\omega_\Delta(\cdot,0)\rangle
    \;\sim\; C_{\omega\omega}(0)\,e^{-t/\tau}, \nonumber \\
    \qquad \Longrightarrow \qquad \zeta \;\approx\; \tau^{-1},
    \label{eq:tau_time}
\end{align}
with the decorrelation time defined by the Green–Kubo–type estimator
\begin{equation}
    \tau \;=\; \frac{1}{C_{\omega\omega}(0)}\int_{0}^{\infty} C_{\omega\omega}(t)\,dt.
    \label{eq:tau_GK}
\end{equation}
In the spectral form, one may define a band-restricted (or shell-averaged) time scale
\begin{equation}
    \tau(k) \;=\; \frac{1}{C_{\omega\omega}(k,0)}\int_{0}^{\infty} C_{\omega\omega}(k,t)\,dt,
    ~~\zeta(k)\;\approx\;\tau(k)^{-1},
    \label{eq:tau_k}
\end{equation}
and for a narrow energetic band $k\in[k_1,k_2]$ use $\zeta \approx \tau(k_\star)^{-1}$ with a representative $k_\star$.
This links the probabilistic estimate to the eddy-viscosity estimate: if $S\approx \nu_t \nabla^2\omega_\Delta$ on the band, then 
$\zeta\approx \nu_t k_\star^2 \approx \tau(k_\star)^{-1}$, which yields $\nu_t \approx [k_\star^2 \tau(k_\star)]^{-1}$.
\emph{Assumptions.} Relations \eqref{eq:tau_time}–\eqref{eq:tau_k} rely on stationarity and an effective Markovian reduction 
(Mori–Zwanzig memory kernel approximated by a delta kernel~\cite{zwanzig2001}), which is consistent with the replacement of the memory integral with a constant gain $\zeta$ over the resolved band.

\section{Derivation of the recursive curl relation}\label{appendix:derivrecurcur}
\paragraph*{Conventions and 2D embedding.}
We work in 2D but embed fields in $\mathbb{R}^3$ by considering the vorticity as
$\bm\omega(\bm r)=\omega(\bm r)\,\bm e_z$.
For $\phi:\mathbb{R}^d\to\mathbb{C}$, the curl and Fourier transform satisfy
\[
\widehat{\nabla\times(\phi\,\bm e_z)}(\bm k)
= i\,\bm k\times\big(\widehat{\phi}\,\bm e_z\big)
= \big(i\,\bm k\times\bm e_z\big)\,\widehat{\phi}(\bm k).
\]
We use $L^2$-sense for Fourier transforms and $F\in L^1$ such that the derivatives commute with filtering.
\paragraph*{From self-convolution to curl in Fourier space.}
Starting from the self-convolution closure,
\begin{align}
\overline{\omega}(\bm r)=\int_{\mathbb{R}^d}\zeta\,G(\bm r-\bm r')\,\omega(\bm r')\,\mathrm d\bm r', \nonumber \\
\qquad\Longleftrightarrow\qquad
\widehat{\overline{\omega}}(\bm k)=\zeta\,\widehat G(\bm k)\,\hat\omega(\bm k), \nonumber
\end{align}
we take the curl (of the embedded vector $\overline{\omega}\,\bm e_z$) and transform:
\[
i\,\bm k\times\widehat{\overline{\omega}}(\bm k)\,\bm e_z
= \zeta\,i\,\bm k\times\big(\widehat G(\bm k)\,\hat\omega(\bm k)\,\bm e_z\big).
\]
Equivalently (dropping the explicit $\bm e_z$ symbolically),
\begin{equation}
i\,\bm k\times\widehat{\overline{\omega}}(\bm k)
= \zeta\,\big[i\,\bm k\times\widehat G(\bm k)\big]\,\hat\omega(\bm k).
\label{eq:A1}
\end{equation}
\paragraph*{Two–vector basis at fixed $\bm k$.}
For each fixed $\bm k\neq\bm 0$, two vectors
\[
\bm e_z,\qquad \bm u(\bm k):=i\,\bm k\times\bm e_z
\]
span the subspace relevant to the left-hand side of \eqref{eq:A1}.
Hence, any vector of the form $\big[i\,\bm k\times\widehat G(\bm k)\big]\,\hat\omega(\bm k)$
can be uniquely decomposed as
\begin{equation}
\big[i\,\bm k\times\widehat G(\bm k)\big]\,\hat\omega(\bm k)
= \widehat{\mathcal H}(\bm k)\,\hat\omega(\bm k)\,\bm e_z
  + \widehat{\mathcal F}(\bm k)\,\bm u(\bm k)\,\hat\omega(\bm k),
\label{eq:A2}
\end{equation}
for scalar multipliers $\widehat{\mathcal H},\widehat{\mathcal F}$ depending on $\bm k$.
Projecting onto this basis yields the following explicit formulae:
\begin{align}
\widehat{\mathcal H}(\bm k)
&= \bm e_z\cdot\Big(i\,\bm k\times\widehat G(\bm k)\,\bm e_z\Big),
\nonumber \\
\widehat{\mathcal F}(\bm k)
&= \frac{\bm u(\bm k)\cdot\Big(i\,\bm k\times\widehat G(\bm k)\,\bm e_z\Big)}
         {\|\bm u(\bm k)\|^2}
= \frac{\bm u(\bm k)\cdot\Big(i\,\bm k\times\widehat G(\bm k)\,\bm e_z\Big)}
         {|\bm k|^2}.\nonumber
\end{align}
Here, we have used $\|\bm u(\bm k)\|^2=|\bm k|^2$.
\paragraph*{Recursive curl relation.}
By substituting \eqref{eq:A2} into \eqref{eq:A1} and using
$\bm u(\bm k)\,\hat\omega(\bm k) = i\,\bm k\times\hat\omega(\bm k)$, we obtain
\[
i\,\bm k\times\widehat{\overline{\omega}}(\bm k)
= \widehat{\mathcal H}(\bm k)\,\hat{\omega}(\bm k)
  + \widehat{\mathcal F}(\bm k)\,\big[i\,\bm k\times\hat{\omega}(\bm k)\big],
\]
which is the recursive relation \eqref{eq:reccursiverelation} in the main text.
Here, $\widehat{\mathcal H}$ represents the ``direct term,'' while $\widehat{\mathcal F}$ represents the
``feedback term'' acting on $i\,\bm k\times\hat\omega$.
\paragraph*{Special case: scalar $G$.}
If $G$ is a scalar kernel (i.e., $\widehat G(\bm k)$ is scalar-valued), then
$i\,\bm k\times\widehat G(\bm k)\,\bm e_z = \widehat G(\bm k)\,\bm u(\bm k)$, and hence
\[
\widehat{\mathcal H}(\bm k) = 0,
\qquad
\widehat{\mathcal F}(\bm k) = \zeta\,\widehat G(\bm k).
\]
Even in this reduced form, the second term on the right-hand side still contains
$i\,\bm k\times\hat\omega$ recursively, making small-to large-scale coupling explicit.
even for the linear operators.

\section{Approximate band-limited equivalence}\label{app:approx}
\HL{The exact theorem (Theorem~2.1) assumes perfectly band-limited inputs and exact equality
of the symbols on the resolved band. In practice, numerical fields are only approximately
band limited and filters possess finite transition regions. The following lemma provides a
simple $L^2$ error estimate in this setting.}

\HL{\begin{lemma}[Approximate band-limited equivalence]
Let $K_1, K_2 \in L^1(\mathbb{R}^d)$ and let $\mathcal{B}_\Delta \subset \mathbb{R}^d$ be a measurable
set of finite measure. For any $\omega \in L^2(\mathbb{R}^d)$ with Fourier transform
$\widehat{\omega}$, we have
\begin{align}
\label{eq:approx_equiv}
\|(T_{K_1} - T_{K_2})\omega\|_{L^2}
\;\le\; 
\|\widehat{K}_1 - \widehat{K}_2\|_{L^\infty(\mathcal{B}_\Delta)}\,
\|\widehat{\omega}\|_{L^2(\mathcal{B}_\Delta)} \nonumber \\
\;+\;
\bigl(\|K_1\|_{L^1} + \|K_2\|_{L^1}\bigr)\,
\|\widehat{\omega}\|_{L^2(\mathcal{B}_\Delta^c)}.
\end{align}
In particular, if $\widehat{K}_1$ and $\widehat{K}_2$ are close on $\mathcal{B}_\Delta$ and
$\|\widehat{\omega}\|_{L^2(\mathcal{B}_\Delta^c)}$ is small (i.e.\ $\omega$ is approximately band
limited), then $T_{K_1}$ and $T_{K_2}$ act approximately equivalently on $\omega$ in $L^2$.
\end{lemma}

\begin{proof}
Write $\widehat{\omega} = \widehat{\omega}_\Delta + \widehat{\omega}_{\Delta^c}$ with
$\widehat{\omega}_\Delta := \widehat{\omega}\,\mathbf{1}_{\mathcal{B}_\Delta}$ and
$\widehat{\omega}_{\Delta^c} := \widehat{\omega}\,\mathbf{1}_{\mathcal{B}_\Delta^c}$.
By linearity,
\[
(T_{K_1} - T_{K_2})\omega
=
(T_{K_1} - T_{K_2})\omega_\Delta
+
(T_{K_1} - T_{K_2})\omega_{\Delta^c},
\]
where $\omega_\Delta$ and $\omega_{\Delta^c}$ are the inverse Fourier transforms of
$\widehat{\omega}_\Delta$ and $\widehat{\omega}_{\Delta^c}$, respectively.

Using the Plancherel theorem and the fact that
$\widehat{K}_j \in L^\infty(\mathbb{R}^d)$ for $K_j \in L^1(\mathbb{R}^d)$, we obtain
\begin{align}
\|(T_{K_1} - T_{K_2})\omega_\Delta\|_{L^2}
=
\|(\widehat{K}_1 - \widehat{K}_2)\,\widehat{\omega}_\Delta\|_{L^2} \nonumber \\
\le
\|\widehat{K}_1 - \widehat{K}_2\|_{L^\infty(\mathcal{B}_\Delta)}\,
\|\widehat{\omega}\|_{L^2(\mathcal{B}_\Delta)}. \nonumber
\end{align}

For the second term, Young's convolution inequality yields
\[
\|T_{K_j}\omega_{\Delta^c}\|_{L^2}
\le \|K_j\|_{L^1}\,\|\omega_{\Delta^c}\|_{L^2},
\qquad j = 1,2.
\]
Therefore,
\[
\|(T_{K_1} - T_{K_2})\omega_{\Delta^c}\|_{L^2}
\le
\bigl(\|K_1\|_{L^1} + \|K_2\|_{L^1}\bigr)\,\|\omega_{\Delta^c}\|_{L^2}.
\]
Plancherel's theorem gives
$\|\omega_{\Delta^c}\|_{L^2} = \|\widehat{\omega}\|_{L^2(\mathcal{B}_\Delta^c)}$.
Combining the bounds for $\omega_\Delta$ and $\omega_{\Delta^c}$ yields
\eqref{eq:approx_equiv}.
\end{proof}

\begin{remark}
Estimate~\eqref{eq:approx_equiv} shows that the exact band-limited equivalence is stable
under small perturbations of the symbols on $\mathcal{B}_\Delta$ and under small spectral leakage
outside $\mathcal{B}_\Delta$. In particular, if $\widehat{K}_1 = \widehat{K}_2$ on $\mathcal{B}_\Delta$ and
$\widehat{\omega}$ is supported in an enlarged band containing $\mathcal{B}_\Delta$ with rapidly
decaying energy outside $\mathcal{B}_\Delta$, then $(T_{K_1}-T_{K_2})\omega$ is controlled entirely
by the out-of-band energy. This situation is typical for numerical filters with smooth transition
regions and for band-pass constructions used in large-eddy simulation and teacher--student
tests.
\end{remark}
}


\section*{Acknowledgements}
This study was supported by JSPS KAKENHI Grant Number 22K14177 and JST PRESTO, Grant Number JPMJPR23O7.
The author acknowledges the stimulating discussions with the members of the PRESTO program, which helped increase awareness of the present problem and its significance. 
The author acknowledges the use of ChatGPT (developed by OpenAI) and DeepL for initial language writing assistance during the manuscript preparation. 
All AI-assisted content was reviewed and revised by the author. 
The author would like to thank Editage (www.editage.jp) for English language editing of the final version of the manuscript. 
The author would like to express sincere gratitude to his family members for their unwavering moral support and encouragement.

\section*{Data availability statement}
My manuscript has no associated data.

\section*{Declaration of interests}
The authors report no conflict of interest.

\bibliographystyle{aipnum4-1}
\bibliography{references}
\end{document}